\documentclass[11pt]{amsart}
\usepackage{geometry}                
\usepackage{graphicx}
\usepackage{amssymb}
\usepackage{epstopdf}
\usepackage{amsmath}
\usepackage{mathrsfs}
\usepackage[all]{xy}
\usepackage{color}
\usepackage[colorlinks=true, pdfstartview=FitV,linkcolor=blue,citecolor=blue,urlcolor=blue]{hyperref}

\newtheorem {theorem}    {Theorem}[section]

\newtheorem {lemma}      [theorem]    {Lemma}
\newtheorem {corollary}  [theorem]    {Corollary}
\newtheorem {prop}[theorem]    {Proposition}

\numberwithin{equation}{section}

\renewcommand{\u}{\underline}

\title{Archimedean zeta integrals on $U(n,1)$}

\author{Bingchen Lin}
\address{School of Mathematics, Sichuan University, Chengdu 610065, P.R. China}
\email{malinbc@scu.edu.cn}

\author{Dongwen Liu}
\address{School of Mathematical Science, Zhejiang University, Hangzhou 310027, P.R. China}
\email{maliu@zju.edu.cn}


\begin{document}
\maketitle

\begin{abstract}
For a dual pair of unitary groups with equal size, zeta integrals arising from Rallis inner product formula give the central values of certain automorphic $L$-functions. In this paper we explicitly calculate archimedean zeta integrals of this type for $U(n,1)$, assuming that the corresponding archimedean component of the automorphic representation is a holomorphic discrete series.
\end{abstract}

\section{Introduction}

In order to obtain deep arithmetic applications in the theory of automorphic forms, it is often necessary to have explicit computable results at each place of a number field. This paper is concerned with certain archimedean zeta integrals on unitary groups and central $L$-values, which arise from theta correspondence of cuspidal automorphic representations. We shall briefly explain the motivation and background of this paper, following \cite{H, HLS, L2}.

Let $F^+$ be a totally real number field, $F$ a totally imaginary quadratic extension of $F^+$, ${\bf A}={\bf A}_{F^+}$ the adele ring of $F^+$. Let $V$ (resp. $V'$) be a hermitian (resp. skew-hermitian) vector space of dimension $n+1$ over $F$, and $W=V\otimes_{F}V'$, a
symplectic space over $F^+$. Fixing an additive character $\psi$ and a complete polarization $W=X\oplus Y$, we have the Schr\"{o}dinger model of the
oscillator representation $\omega_\psi$ of $\widetilde{Sp}(W)({\bf A})$, realized on the space ${\mathcal S}(X({\bf A}))$ of Schwartz-Bruhat functions on $X({\bf A})$.
Let $G=U(V),$ $G'=U(V')$. By choosing a global splitting character $\chi$ of ${\bf A}_{F}^\times/F^\times$ as in \cite{HLS}, $\omega_\psi$ then defines an oscillator representation $\omega_{V,V',\chi}$ of $G({\bf A})
\times G'({\bf A})$ on ${\mathcal S}(X({\bf A}))$. As usual, for $\phi\in {\mathcal S}(X({\bf A}))$ we have the theta lifting $f\mapsto \theta_\phi(f)$ for a cusp form $f$ on $G(F^+)\backslash G({\bf A})$.

Let $\pi$ be a cuspidal automorphic representations of $G$, $f\in\pi$, $\tilde{f}\in\pi^\vee$.
Let $H=U(V\oplus (-V))$, $i_V: G\times G\hookrightarrow H$
be the natural inclusion, following the doubling method. The Piatetski-Shapiro-Rallis zeta integral is then defined by
\begin{equation}
Z(s, f, \tilde{f},\varphi,\chi)=\int_{(G\times G)(F^+)\backslash (G\times G)({\bf A})}E(i_V(g,\tilde{g}), s, \varphi, \chi)f(g)\tilde{f}(\tilde{g})\chi^{-1}\big(\det(\tilde{g})\big)dg d\tilde{g},
\end{equation}
where $E(\cdot, s,\varphi, \chi)$ is the Eisenstein series on $H({\bf A})$ as in \cite[\S 1]{H}, and $\varphi=\varphi(s)$ is a section of a degenerate principal series
$I_{n+1}(s,\chi)$ varying in $s$. This integral converges absolutely for $\textrm{Re }s\gg0$ and admits an Euler expansion if $\varphi$, $f$ and $\tilde{f}$ are factorizable. In this case,
for $S$ a sufficiently large finite set of places of $F^+$ including archimedean ones, one has
\[
Z(s, f, \tilde{f}, \varphi, \chi)=\prod_{v\in S}Z(s, f_v, \tilde{f}_v, \varphi_v, \chi_v)d_{n+1}^S(s)^{-1}L^S(s+\frac{1}{2},\pi, St,\chi),
\]
where $L^S(s+\frac{1}{2},\pi, St,\chi)$ is the partial $L$-function of $\pi$ twisted by $\chi$, attached to $2(n+1)$-dimensional standard representation of the $L$-group, and
$d_{n+1}^S(s)$ is a product of certain partial $L$-functions attached to the extension $F/F^+$ as in \cite{H}. Take $\phi=\otimes_v\phi_v\in {\mathcal S}(X({\bf A}))$ and $\varphi=\delta(\phi\otimes\bar{\phi})$ in the notation of \cite[p.182]{L2}. Then after proper
normalization the {\it Rallis inner product
formula} can be written as
\begin{equation}
\langle\theta_\phi(f),\theta_{\bar{\phi}}(\bar{f})\rangle=\prod_{v\in S}Z(0, f_v,\bar{f}_v,\varphi_v,\chi_v)d_n^S(0)^{-1}L^S(\frac{1}{2},\pi, St,\chi),
\end{equation}

The central $L$-value $L(\frac{1}{2},\pi, St,\chi)$ is of great arithmetic interest and it is quite useful to have explicit local value at each place. In \cite{HLS} under certain assumptions it is shown that
$L^S(\frac{1}{2},\pi, St,\chi)\geq 0$ for any finite set $S$ of places of $F^+$.
As explained in \cite[\S2]{L2}, one has
\[
Z(0,f_v, \bar{f}_v, \varphi_v,\chi_v)=\int_{G(F_v^+)}(\omega_{\chi_v} (g)\phi_v, \phi_v) (\pi_v(g)f_v, f_v)dg,
\]
which integrates matrix coefficient of the oscillator representation against that of $\pi_v$. From now on we assume that $v$ is real, $\pi_v$ is in the discrete series, $\phi_v$ is in the space of joint harmonics, and we replace
$(\pi_v(g)f_v, f_v)$ by a canonical matrix coefficient $\psi_{\pi_v}(g)$ of $\pi_v$ (see Section \ref{s4}). The aim of this paper is to explicitly compute the archimedean zeta integral
\begin{equation} \label{zeta'}
\int_{G(F_v^+)}(\omega_{\chi_v}(g)\phi_v,\phi_v)\cdot \psi_{\pi_v}(g)dg
\end{equation}
in the case that $G(F_v^+)$ is the real unitary group $U(n,1)$ and $\pi_v$ is a holomorphic discrete series. We mention that the cases $U(1,1)$ and $U(2,1)$ were solved completely in \cite{Lin} and \cite{Liu} respectively. However for $U(n,1)$ when $\pi_v$ is a general discrete series, this problem seems to out of reach at the moment.

The main results of this paper can be formulated as follows. Fix an additive character $\psi$ of ${\bf R}$. Let $V$ be an $(n+1)$-dimensional complex
Hermitian space, and let $G$ be the unitary group attached to $V$. For each
complex skew-Hermitian space $V'$, the group $G$ is a subgroup of
the real symplectic group $Sp(V\otimes_{\bf C} V')$ as usual. Define the metaplectic double cover $\widetilde{G}$
of $G$ to be the double cover of $G$ induced by the metaplectic double
cover $\widetilde{Sp}(V\otimes_{\bf C}V')\to Sp(V
\otimes_{\bf C}V')$. This is independent of $V'$. 
Let $\pi_\lambda$ be the genuine discrete series
representation of $\widetilde{G}$
with Harish-Chandra parameter $\lambda:=(\lambda_1,\ldots, \lambda_{n+1})$. By theta dichotomy for real unitary groups \cite{P} and a result in \cite{L1} on discrete spectrum of local theta correspondence, up to isometry there exists a unique $(n+1)$-dimensional skew-Hermtian space $V'$ such that $\pi^\vee_\lambda$ occurs as a subrepresentation of $\omega_{V,V',\chi}$. Let $P_\lambda:\omega_{V,V',\chi}\to \omega_{V,V',\chi}$ be the orthogonal projection to the $\pi^\vee_\lambda$-isotypic subspace.
Fix a maximal compact subgroup $K$ of $G$, which induces a maximal compact
subgroup $\widetilde{K}$ of $\widetilde{G}$.
Denote by $\tau^\vee_\lambda$ the lowest $\widetilde{K}$-type of $\pi^\vee_\lambda$. Then there is a positive number $c_{\psi, V, \lambda}$ such that
\[
\| P_\lambda(\phi)\| = c_{\psi, V, \lambda} \|\phi\|
\]
for all $\phi$ in the  $\tau^\vee_\lambda$-isotypic subspace of the space of joint harmonics
(with respect to $K$ and an arbitrary maximal compact subgroup of the unitary
group attached to $V'$). The constant $c_{\psi,V,\lambda}$ is 1 when either $V$ or $V'$ is anisotropic. The main result of this paper is equivalent to an explicit
calculation of $c_{\psi, V, \lambda}$ when $V$ is of signature $(n,1)$, $\psi$ is chosen to be $\psi_a: t\mapsto e^{2\pi i a t}$ for some $a>0$, and $\pi_\lambda$ is holomorphic. In this case we list the explicit values of $c_{\psi,V,\lambda}$ below (Corollary \ref{7.3}).

\begin{theorem}
Follow above notations, assume that $V$ has signature $(n,1)$, $\psi=\psi_a$ for some $a>0$ and $\pi_\lambda$ is holomorphic. Let
\[
\Lambda=\lambda+(-\frac{n}{2}+1,-\frac{n}{2}+2,\ldots, \frac{n}{2}, -\frac{n}{2}).
\]
Let $\alpha$'s, $\beta$'s, $\gamma$, $p, q$ below stand for non-negative integers with $p+q=n+1$. Then
\\
(i) if $\Lambda=[(\alpha_1,\ldots,\alpha_n)+\det^{-(1-n)/2}]\otimes[\gamma+\det^{(1-n)/2}]$ with $\alpha_1\geq\cdots\geq \alpha_n\geq\gamma+2$, then
\[
c^2_{\psi,V,\lambda}=\prod^n_{i=1}\frac{\alpha_i-i+n-1-\gamma}{\alpha_i-i+n};
\]
(ii) if $\Lambda=[(\alpha_1,\ldots,\alpha_{q-1},-\beta_p,\ldots,-\beta_1)+\det^{-(p-q)/2}]\otimes[-\gamma+\det^{(p-q)/2}]$ with $\alpha_1\geq\cdots\geq\alpha_{q-1}\geq -\beta_p\geq\cdots\geq -\beta_1\geq -\gamma+2p$, then
\[
c^2_{\psi,V,\lambda}=\prod^n_{i=1}\frac{\gamma+i-\delta_i-2p}{\gamma+i-p},
\]
where $(\delta_1,\ldots,\delta_n):=(\beta_1,\ldots,\beta_p,-\alpha_{q-1},\ldots,-\alpha_1)$.
\end{theorem}

The organization of the paper is as follows. In section \ref{s2} we give the pair of weights appearing in the local theta correspondence. Section \ref{s3} describes the structure and measure of the real Lie group $U(n,1)$. Section \ref{s4} deals with the canonical matrix coefficient of a holomorphic discrete series following \cite{G}. Sections \ref{s5} and \ref{s6} are concerned with the matrix coefficient of oscillator representation, which is calculated using joint harmonics. In section \ref{s7} we combine previous results and apply the technique of \cite{G} to evaluate the zeta integral. 

We remark that the method of this paper should be applicable to general $U(p,q)$, at least when one of the components of the lowest $\widetilde{K}$-type is one-dimensional. Furthermore, it also brings us some enlightenment to study certain period integrals for unitary groups.

{\bf Notations.} Let $1_n$ and $0_n$ be the $n\times n$ identity matrix and zero matrix respectively. Let $1_{p,q}$ stand for the square matrix
\[
\begin{pmatrix} 1_p & 0 \\ 0 & -1_q\end{pmatrix}.
\]
In this paper, $U(p,q)$ is the real unitary group of the hermitian or skew-hermitian form represented by
the matrix $1_{p,q}$ or $i1_{p,q}$, where $i=\sqrt{-1}$, and $Sp_{2N}({\bf R})$ is the isometry group of the real symplectic form represented by the matrix
\[
\begin{pmatrix} 0 & 1_N \\ -1_N & 0\end{pmatrix}.
\]
For a complex matrix $g$, let ${}^tg$ be its transpose, and $g^*={}^t\bar{g}$ be the complex conjugate transpose. For a field $k$, $M_n(k)$ is the set of $n\times n$ matrices with entries in $k$.  We usually regard vectors in $k^n$ as column vectors, unless otherwise specified. For $u, v\in k^n$,  as usual $u\cdot v$ stands for their dot product, and $|u|^2=u\cdot \bar{u}$ if $k={\bf R}$ or ${\bf C}$.

{\bf Acknowledgement.} Both authors would like to thank Professor J.-S. Li for suggesting this problem. This joint work was started during the workshop ``Automorphic Forms, Geometry and Representation Theory" held at Zhejiang University in July 2015. Both authors thank the hospitality of the organizers.

\section{Pair of weights} \label{s2}

Let $G=U(n,1)$ be the unitary group of a complex hermitian space of signature $(n,1)$. The absolute root system of $G_{\bf{C}}=GL(n+1,\bf{C})$ is of type $A_n$.
Fix the maximal compact subgroup $K=U(n)\times U(1)$, the set of compact positive roots $\Delta^+_c=\{e_i-e_j: 1\leq i<j\leq n\}$, and the set of positive roots
$\Delta^+=\{e_i-e_j: 1\leq i<j\leq n+1\}$ that contains $\Delta_c^+$.

We consider $\Delta_c^+$-dominant Harish-Chandra parameters of genuine discrete series of $\widetilde{G}$. Those of holomorphic discrete series  are in fact $\Delta^+$-dominant, i.e. strictly decreasing $(n+1)$-tuples $\lambda=(\lambda_1,\ldots,\lambda_{n+1})$ of half-integers. The corresponding lowest $\widetilde{K}$-type is given by
the Blattner parameter
\begin{equation}\label{hc}
\Lambda=\lambda+\rho-2\rho_c=\lambda+(-\frac{n}{2}+1,-\frac{n}{2}+2,\ldots, \frac{n}{2}, -\frac{n}{2}),
\end{equation}
where $\rho$ (resp. $\rho_c$) is the half sum of all positive (resp. compact positive) roots.

Consider the dual pair $(G,G')=(U(n,1),U(p,q))\hookrightarrow Sp_{2N}(\bf{R})$, where $p+q=n+1$ and $N=(n+1)^2$. Fix the additive character $\psi: t\mapsto e^{2\pi it}$ of ${\bf R}$, and consider the oscillator representation $\omega_\psi$ of $\widetilde{Sp}_{2N}({\bf R})$. Take an irreducible $\widetilde{K}\times \widetilde{K}'$-module
\[
\mathcal{H}_{\Lambda^\vee,\Lambda'}\cong \sigma_{\Lambda^\vee}\otimes \sigma_{\Lambda'}
\]
that occurs in the space of joint harmonics of $\omega_\psi$, where $\Lambda^\vee$ and $\Lambda'$ are the highest weights of $\sigma_{\Lambda^\vee}$
and $\sigma_{\Lambda'}$ respectively. It is well-known that $\mathcal{H}_{\Lambda^\vee, \Lambda'}$ occurs with multiplicity one, and moreover $\Lambda^\vee$ and $\Lambda'$ determine each other. Let $\sigma_\Lambda$ be the contragradient of $\sigma_{\Lambda^\vee}$, which has highest weight $\Lambda$.

Assume that $\sigma_\Lambda$ is the lowest $\widetilde{K}$-type of the holomorphic discrete series $\pi_\lambda$ of $\widetilde{G}$ so that $\lambda$ and $\Lambda$ are related by (\ref{hc}), and that the theta lifting $\pi'=\theta(\pi_{\lambda}^\vee)$ of $\pi_\lambda^\vee$ is a non-zero discrete series of $\widetilde{G}'$.

The  Harish-Chandra parameter of the anti-holomorphic discrete series $\pi_\lambda^\vee$ is $\lambda^\vee=(-\lambda_n,\ldots, -\lambda_1, -\lambda_{n+1})$, and one has
\[
\Lambda^\vee=\lambda^\vee+(-\frac{n}{2},-\frac{n}{2}+1,\ldots, \frac{n}{2}-1, \frac{n}{2}).
\]
 Let $a$ and $b$ be the number of non-negative entries in $(-\lambda_n,\ldots, -\lambda_1)$ and $(-\lambda_{n+1})$ respectively. Then by \cite{L1}, above assumption requires that
\[
\lambda_n> \lambda_{n+1}\quad\textrm{and}\quad
p=a-b+1.
\]
Let us write
\begin{equation}\label{lamv}
\Lambda^\vee=[(\beta_1,\ldots,\beta_l,0,\ldots,0,-\alpha_k,\ldots,-\alpha_1)+\textrm{det}^{(p-q)/2}]
\otimes [m+\textrm{det}^{-(p-q)/2}],
\end{equation}
where $\alpha_1\geq\cdots\geq \alpha_k>0$, $\beta_1\geq\cdots\geq \beta_l>0$. Then
\[
\Lambda=[(\alpha_1,\ldots,\alpha_k,0,\ldots,0,-\beta_l,\ldots,-\beta_1)+\textrm{det}^{-(p-q)/2}]
\otimes[-m+\textrm{det}^{(p-q)/2}].
\]
We have two cases.

Case (i): $b=0$. Then $\lambda_n>\lambda_{n+1}>0$, which implies that $a=p-1=0$ hence $p=1$, $q=n$, i.e. $G'=U(1,n)$. The first entry of $\Lambda^\vee$ is
\[
-\lambda_n-\frac{n}{2}<\frac{p-q}{2}=\frac{1-n}{2},
\]
which by (\ref{lamv}) implies that $l=0$, $k=n$. Let
\[
\gamma:=-m=\lambda_{n+1}-\frac{1}{2}\geq 0.
\]
By the formulas for the pair of weights $\Lambda^\vee$, $\Lambda'$ in \cite{L1}, we see that
\begin{equation}\label{w1}
\left\{\begin{array}{ll}
\Lambda^\vee=[(-\alpha_n,\ldots,-\alpha_1)+\textrm{det}^{(1-n)/2}]
\otimes [-\gamma+\textrm{det}^{-(1-n)/2}],
\\
\Lambda'=[-\gamma+\textrm{det}^{(n-1)/2}]\otimes[(-\alpha_n,\ldots,-\alpha_1)+\textrm{det}^{-(n-1)/2}].\end{array}\right.
\end{equation}
The condition $\lambda_n>\lambda_{n+1}$ reads
\[
\alpha_n\geq\gamma+2.
\]

Case (ii): $b=1$. Then $\lambda_{n+1}\leq 0$, $a=p$. Let
\[
\gamma:=m=-\lambda_{n+1}+\frac{n}{2}+\frac{p-q}{2}>0.
\]
Again by \cite{L1} we have
\[
\Lambda'=[(\beta_1,\ldots,\beta_l,0,\ldots,0)+\textrm{det}^{(n-1)/2}]\otimes[(\gamma,0,\ldots,0, -\alpha_k,\ldots,-\alpha_1)+\textrm{det}^{(1-n)/2}].
\]
Note that the obvious constraints $l\leq p$, $k+1\leq q$ apply. For convenience let us define $\beta_i=0$, $\alpha_j=0$ for $l<i\leq p$ and $k<j\leq q-1$, so that we may write
\begin{equation}\label{w2}
\left\{\begin{array}{ll}
\Lambda^\vee=[(\beta_1,\ldots,\beta_p, -\alpha_{q-1},\ldots,-\alpha_1)+\textrm{det}^{(p-q)/2}]
\otimes [\gamma+\textrm{det}^{-(p-q)/2}],
\\
\Lambda'=[(\beta_1,\ldots,\beta_p)+\textrm{det}^{(n-1)/2}]\otimes[(\gamma, -\alpha_{q-1},\ldots,-\alpha_1)+\textrm{det}^{-(n-1)/2}].
\end{array}\right.
\end{equation}
 The condition $\lambda_n> \lambda_{n+1}$ reads
\[
-\beta_1\geq-\gamma+2p.
\]

\section{Structure of $G$} \label{s3}

Let $\frak{g}=\frak{u}(n,1)$ be the Lie algebra of $G$, and $\frak{g}=\frak{k}\oplus\frak{p}$ be the Cartan decomposition with respect to the Cartan involution $\theta(X)=-X^*$. Let $\frak{a}$ be the maximal abelian subalgebra of $\frak{p}$, which is one-dimensional and spanned by, say,
\[
H=E_{1,n+1}+E_{n+1,1},
\]
where $E_{ij}$ is the elementary matrix with $1$ on the $(i,j)$-entry and $0$ everywhere else.
Let
\[
a_t=\exp(tH)=\begin{pmatrix}
\cosh t & 0 & \sinh t \\ 0 & 1_{n-1} & 0\\
\sinh t & 0 & \cosh t
\end{pmatrix}.
\]
The Cartan decomposition of $G$ is  $G=C\cdot K \cong C\times K$, where
\[
C=\{g\in G: g=g^* \textrm{ is positive-definite hermitian}\}.
\]
We normalize the measure on $K=U(n)\times U(1)$ so that the masses of $U(n)$ and $U(1)$ are both equal to 1.
The set $C$ can be parametrized by
\begin{equation}\label{c}
D_{n,1}\to C,\quad z\mapsto h_z=\begin{pmatrix} (1_n-zz^*)^{-1/2} & z(1-z^*z)^{-1/2}\\
(1-z^*z)^{-1/2}z^* & (1-z^*z)^{-1/2}\end{pmatrix}
\end{equation}
where $D_{n,1}$ is the classical domain
\[
D_{n,1}=\{z\in {\bf C}^n: 1_n-zz^*\textrm{ is positive definite}\}.
\]
$G$ acts on $D_{n,1}$ by generalized fractional linear transformations, and we fix the  invariant measure on $D_{n,1}$ to be
\[
d^*z=\frac{dz}{(1-z^*z)^{n+1}}=\frac{dz}{\det(1_n-zz^*)^{n+1}},
\]
where $dz$ is the product of the usual additive Haar measures.

One may further parametrize $D_{n,1}$ by $z=x\underline{r}y$, where $x\in U(n)$, $y\in U(1)$, and $\u{r}={}^t(r, 0, \ldots, 0)$ with $-1<r<1$. If we write $r=\tanh t$, $t\in{\bf R}$, then substituting this parametrization into (\ref{c}) yields
\begin{equation}\label{mea}
h_z=k_z a_t k_z^{-1},\quad k_z=\begin{pmatrix} x & 0 \\ 0 & y^*\end{pmatrix}\in K.
\end{equation}

\section{Holomorphic discrete series} \label{s4}

We shall briefly review the treatment in \cite{G}. Recall that $\frak{g}$ is the Lie algebra of $G$, and let $\frak{g}_{\bf C}$ be its complexification. Let
\[
\frak{p}_+=\left\{\begin{pmatrix} 0_n & * \\ 0 & 0\end{pmatrix}\in\frak{g}_{\bf C}\right\},\quad
\frak{p}_-=\left\{\begin{pmatrix} 0_n & 0 \\ * & 0\end{pmatrix}\in\frak{g}_{\bf C}\right\}
\]
and $N_\pm = \exp\frak{p}_\pm$. Then one has the Harish-Chandra decomposition
\[
G\subset N_+\cdot K_{\bf C}\cdot N_-\subset G_{\bf C}.
\]
Let $\pi=\pi_\lambda$ be a holomorphic discrete series with lowest $K$-type $\sigma=\sigma_\Lambda$. In \cite{G} it is shown that the canonical $K$-conjugation invariant matrix coefficient of $\pi$ is given by
\[
\psi_\pi(g)=\psi_\pi (n_+\theta n_-)=\textrm{tr } \sigma(\theta) \in \sigma  \otimes \sigma\subset \pi\otimes \pi^\vee \subset L^2(G)
\]
if $g=n_+\theta n_-$ under the Harish-Chandra decomposition. Here we use the holomorphic extension of $\sigma$ to $K_{\bf C}$. We remark that $\psi_\pi$ is equivalent to the canonical matrix coefficient considered in \cite{F-J, HLS, L1, Liu}.

Recall the Cartan decomposition $g=h_z k$. The Harish-Chandra decomposition
$h_z=n_z^+\theta_z n_z^-$ is
\[
h_z=\begin{pmatrix} 1_n & z \\ 0 & 1\end{pmatrix}\begin{pmatrix} (1_n-zz^*)^{1/2} & 0 \\ 0 &
(1-z^*z)^{-1/2}\end{pmatrix}\begin{pmatrix} 1_n & 0 \\ z^* & 1 \end{pmatrix}.
\]
In particular
\begin{equation}\label{thetaz}
\theta_z=\begin{pmatrix} (1_n-zz^*)^{1/2} & 0 \\ 0 &
(1-z^*z)^{-1/2}\end{pmatrix}.
\end{equation}
Then we have
\begin{equation}\label{dismc}
\psi_\pi(g)=\psi_\pi (n^+_z\theta_z k k^{-1}n^-_z k)=\textrm{tr }\sigma(\theta_z k),
\end{equation}
noting that $K$ normalizes $N_\pm$. Parametrizing $z\in D_{n,1}$ as in Section \ref{s3}, we may write
\[
\theta_z= k_z  \theta_t k_z^{-1},
\]
where $k_z$ is as in (\ref{mea}) and $\theta_t$ is the $K_{\bf C}$-component of $a_t$ under Harish-Chandra decomposition, i.e.
\begin{equation}\label{thetat}
\theta_t= \begin{pmatrix} (\cosh t)^{-1} & 0 & 0 \\ 0 & 1_{n-1} & 0 \\ 0 & 0 & \cosh t\end{pmatrix}.
\end{equation}
Therefore one may further write
\begin{equation}\label{dismc2}
\psi_\pi(g)=\textrm{tr }\sigma(k_z\theta_t k_z^{-1}k )=\textrm{tr }\sigma(\theta_t k_z^{-1}k k_z).
\end{equation}

Finally we remark that by Corollary of \cite[Lemma 23.1]{H-C}, the formal degree $d_\pi$ of a general discrete series $\pi=\pi_\lambda$ is given by
\begin{equation}\label{fd}
d_\pi=C\prod_{1\leq i<j\leq n+1}|\lambda_i-\lambda_j|
\end{equation}
where $C$ is a constant depending on the choice of the Haar measure of $G$.

\section{Fock model} \label{s5}

The smooth model $\omega^\infty_\psi$ of the oscillator representation of $\widetilde{Sp}_{2N}({\bf R})$ can be realized on the Fock space $\mathscr{F}_N$ of entire functions on ${\bf C}^N$ which are square integrable with respect to the hermitian inner product
\[
\langle f, g\rangle_\omega=\int_{{\bf C}^N} f(z)\overline{g(z)} e^{-\pi|z|^2}dz.
\]
The monomials $\left\{\sqrt{\dfrac{\pi^{|\alpha|}}{\alpha!}}z^\alpha, |\alpha|\geq 0\right\}$ forms an orthonormal basis of ${\mathscr F}_N$.
The Harish-Chandra module $\omega^{HC}_\psi$ can be realized as the subspace $\mathscr{P}_N$ of polynomials on ${\bf C}^N$.

Following \cite{F}, introduce the linear map
\begin{equation}\label{gc}
M_{2N}({\bf R})\to M_{2N}({\bf C}),\quad g=\begin{pmatrix} A & B\\ C & D\end{pmatrix}\mapsto g^c=\frac{1}{2}\begin{pmatrix} A+D+i(C-B) &
A-D+i(C+B)\\ A-D-i(C+B) & A+D-i(C-B)\end{pmatrix}.
\end{equation}
Denote by $Sp^c_{2N}$ the image of $Sp_{2N}({\bf R})$. Let $\nu$ be the Fock projective representation of $Sp^c_{2N}$ on ${\mathscr F}_N$. Then for
$g^c= \begin{pmatrix} P & Q\\ \overline{Q} & \overline{P}\end{pmatrix} \in Sp^c_{2N}$,  up to a factor of $\pm 1$ the operator $\nu(g^c)$ is given by
\begin{equation}\label{fockaction}
\left\{
\begin{aligned}
&\nu(g^c)f(z) = \int_{{\bf C}^N}K_{g^c}(z, \bar{w})f(w)e^{-\pi |w|^2} dw,\\
&K_{g^c}(z,\bar{w})= (\det P)^{-\frac{1}{2}} \exp\bigg[\frac{\pi}{2}\big({}^tz\overline{Q}P^{-1}z+2{}^t\bar{w} P^{-1}z -{}^t\bar{w}P^{-1} Q \bar{w}\big)\bigg].
\end{aligned}\right.
\end{equation}
Let $J=1_{n,1}\otimes 1_{p,q}$. We have the embedding
\begin{equation}\label{gembed}
i_G: G\hookrightarrow Sp_{2N}({\bf R}),\quad X+iY\mapsto \begin{pmatrix} X\otimes 1_{n+1} & (Y\otimes 1_{n+1})J\\
-J(Y\otimes 1_{n+1}) & J(X\otimes 1_{n+1})J\end{pmatrix}.
\end{equation}
In spirit of ${\bf C}^N\cong {\bf C}^{n+1}\otimes {\bf C}^{n+1}\cong M_{n+1}({\bf C})$, it is more convenient to label the variables $z_1, \ldots, z_N$ as $z_{11},\ldots, z_{1,n+1}, \ldots, z_{n+1,1},\ldots, z_{n+1,n+1}$.
For instance, if we write the matrix of variables in the block form
\begin{equation}\label{block}
z=(z_{ij})_{i,j=1,\ldots,n+1}=\begin{pmatrix} A_{n\times p} & B_{n\times q} \\
C_{1\times p} & D_{1\times q}\end{pmatrix},
\end{equation}
then from (\ref{gc}-\ref{gembed}), for $k=(x, y)\in K=U(n)\times U(1)$ one has up to $\pm1$
\begin{equation}\label{fockk}
\omega(k) f(z)=(\det x)^{(p-q)/2}(\det y)^{-(p-q)/2}f\begin{pmatrix}
{}^tx A & x^{-1}B \\ y^{-1}C & {}^t y D\end{pmatrix}.
\end{equation}

We need to know the action of $\omega(a_t)$.
We use the notation $\u{z}_i= (z_{i,1},\ldots, z_{i,n+1})$, $i=1,\ldots, n+1$, so that we can write $f(z)=f(\u{z}_1,\ldots, \u{z}_{n+1})\in \mathscr{F}_N$.  As a preliminary step we have

\begin{lemma} \label{4.1}   For $f(z)\in\mathscr{P}_N$,
\begin{align*}
&\omega(a_t)f(z)=(\cosh t)^{-n-1}\exp\left(\pi(\tanh t) \u{z}_1\cdot \u{z}_{n+1}\right) \\
&~ \times \int_{{\bf C}^{n+1}}f(\u{w}_1, \u{z}_2,\ldots, \u{z}_n, (\cosh t)^{-1}\u{z}_{n+1}- (\tanh t) \bar{\u{w}}_1)\exp\left(\pi(\cosh t)^{-1}\u{z}_1\cdot \bar{\u{w}}_1-\pi |\u{w}_1|^2\right)d\u{w}_1.
\end{align*}
\end{lemma}

\begin{proof}
One can show that
\[
i_G(a_t)=\begin{pmatrix} a_t\otimes 1_{n+1} & 0 \\ 0 & a_{-t}\otimes 1_{n+1}\end{pmatrix},\quad i_G(a_t)^c=\begin{pmatrix} P_t & Q_t \\ Q_t & P_t\end{pmatrix}
\]
where
\[
P_t=\begin{pmatrix} \cosh t & 0 & 0 \\ 0 & 1_{n-1} & 0 \\ 0 & 0 & \cosh t\end{pmatrix}\otimes 1_{n+1},\quad Q_t=\begin{pmatrix} 0 & 0 & \sinh t \\ 0 & 0_{n-1} & 0 \\ \sinh t & 0 & 0\end{pmatrix}\otimes 1_{n+1}.
\]
We calculate that
\begin{align*}
&{}^t zQ_t P_t^{-1}z=2(\tanh t) \u{z}_1\cdot\u{z}_{n+1},\quad {}^t\bar{w}P_t^{-1}Q_t\bar{w}=2(\tanh t)\bar{\u{w}}_1\cdot \bar{\u{w}}_{n+1}\\
&{}^t\bar{w}P_t^{-1}z=(\cosh t)^{-1}(\u{z}_1\cdot\bar{\u{w}}_1+\u{z}_{n+1}\cdot\bar{\u{w}}_{n+1})+\sum^n_{i=2} \u{z}_i\cdot \bar{\u{w}}_i.
\end{align*}
The lemma follows from integrating over $\u{w}_2,\ldots, \u{w}_{n+1}$ in (\ref{fockaction}) and applying the following formula, which will be used later as well.
\begin{equation}\label{formula}
\int_{\bf C} z^i\bar{z}^j e^{\pi c \bar{z}-\pi|z|^2}dz=\left\{\begin{array}{ll}
  \dfrac{i!}{(i-j)!}\dfrac{c^{i-j}}{\pi^j}& \textrm{if }i\geq j,\\ 0 & \textrm{if } i<j,
\end{array}\right.
\end{equation}
where $i, j\geq 0$ are integers and $c$ is a constant.
\end{proof}

\section{Joint harmonics} \label{s6}

The notion of joint harmonics was introduced in \cite{Ho}. It is the subspace  $\mathcal{H}\subset \mathscr{P}_N$ annihilated by certain second order differential operators from the the centralizers of $\frak{k}$ and $\frak{k}'$ in $\frak{sp}$, under the action of oscillator representation. We refer the readers to \cite{Ho} for the precise definition.

It is known that $\mathcal{H}$ admits a multiplicity free decomposition
\[
\mathcal{H}\cong \bigoplus \sigma\otimes \sigma'
\]
into irreducible $\widetilde{K}\times \widetilde{K}'$-modules such that $\sigma$ and $\sigma'$ determine each other. Moreover, the lowest $\widetilde{K}$- and $\widetilde{K}'$-type of discrete series correspond to each other under this decomposition.

We consider the subspace of joint harmonics $\mathcal{H}_{\Lambda^\vee,\Lambda'}\cong \sigma_{\Lambda^\vee}\otimes \sigma_{\Lambda'}$  as in Section \ref{s2}. The joint highest weight vector of $\mathcal{H}_{\Lambda^\vee,\Lambda'}$ can be expressed in terms of principal minors. For $i=1,\ldots, n$, let
\[
\Delta_i=\det\begin{pmatrix} z_{11} & \cdots & z_{1i} \\
&\cdots & \\
z_{i1} & \cdots & z_{ii}\end{pmatrix},\quad \Delta_i'=\det\begin{pmatrix}
z_{n-i+1, n-i+2} & \cdots & z_{n-i+1, n+1} \\
& \cdots & \\
 z_{n, n-i+2} & \cdots & z_{n,n+1}\end{pmatrix},
\]
which are determinants of $i\times i$ minors, hence homogeneous polynomials. Then in the two cases of Section \ref{s2}, we have the following harmonic polynomials of joint highest weight, which are unique up to scalar.

Case (i):  We take
\[
\phi(z)=\Delta_1'^{\alpha_1-\alpha_2}\Delta_2'^{\alpha_2-\alpha_3}\cdots \Delta_n'^{\alpha_n}z_{n+1,1}^\gamma.
\]
For any $k\in K$, from (\ref{fockk}) we see that the block $C$ in (\ref{block}), i.e. $z_{n+1,1}$, is the only variable of
$\u{z}_{n+1}$ that appears in $\omega(k)\phi$, but the block $A$, in particular $z_{11}$, does not show up in
$\omega(k)\phi$. This observation together with Lemma \ref{4.1} and (\ref{formula}) gives us
\begin{align*}
\omega(a_t k)\phi(z)&=(\cosh t)^{-n-1}\exp\left(\pi(\tanh t) \u{z}_1\cdot \u{z}_{n+1}\right) \omega(k)\phi\left((\cosh t)^{-1}\u{z}_1,\u{z}_2,\ldots, \u{z}_n,
(\cosh t)^{-1}\u{z}_{n+1}\right)\\
&=(\cosh t)^{-n-1}\exp\left(\pi(\tanh t) \u{z}_1\cdot \u{z}_{n+1}\right) \sigma_{\Lambda^\vee}(b_t k)\phi,
\end{align*}
where we use the extension of $\sigma_{\Lambda^\vee}$ to $K_{\bf C}$, and
\begin{equation}\label{bt}
b_t=\begin{pmatrix} \cosh t & 0 & 0 \\ 0 & 1_{n-1} & 0 \\ 0 & 0 & \cosh t\end{pmatrix}.
\end{equation}
Since $\sigma_{\Lambda^\vee}$-action preserves the degree, and monomials are orthogonal basis, we may use Taylor expansion to drop the factor $\exp\left(\pi(\tanh t) \u{z}_1\cdot \u{z}_{n+1}\right)$ and obtain
\begin{equation}\label{wmc}
\langle\omega(k'a_tk)\phi, \phi\rangle=(\cosh t)^{-n-1}\langle \sigma_{\Lambda^\vee}(k' b_t k)\phi, \phi\rangle
\end{equation}
for any $k, k'\in K$.

Case (ii): We take
\[
\phi(z)=\Delta_1^{\beta_1-\beta_2}\Delta_2^{\beta_2-\beta_3}\cdots \Delta_p^{\beta_p}\Delta_1'^{\alpha_1-\alpha_2}\Delta_2'^{\alpha_2-\alpha_3}\cdots\Delta_{q-1}'^{\alpha_{q-1}}z_{n+1,p+1}^\gamma.
\]
The argument is similar to above. We note that $z_{n+1,p+1}$ is the only variable of $\u{z}_{n+1}$ that appears in $\omega(k)\phi$, while the first $p$ rows of the block $B$ in
(\ref{block}), in particular
$\u{z}_{1,p+1}$, do not show up. The same argument as above gives us
\begin{equation}\label{wmc2}
\langle\omega(k'a_tk)\phi, \phi\rangle=(\cosh t)^{-n-1}\langle \sigma_{\Lambda^\vee}(k' b_t^{-1} k)\phi, \phi\rangle
\end{equation}

We may summarize our results as

\begin{prop}\label{6.1}
Under the assumptions of Section \ref{s2}, for a vector $\phi\in\mathcal{H}_{\Lambda^\vee,\Lambda'}$ of joint highest weight, $k, k'\in K$, one has
\[
\langle\omega(k'a_tk)\phi, \phi\rangle=(\cosh t)^{-n-1}\langle \sigma_{\Lambda^\vee}(k' b_t^{\pm 1} k)\phi, \phi\rangle,
\]
where the $\pm$ sign depends on whether the first (or equivalently, the last) component of $\Lambda^\vee$ is negative or positive.
\end{prop}

In particular, by the Harish-Chandra decomposition $g=h_zk=k_z a_t k_z^{-1}k$, one has
\[
\langle\omega(g)\phi,\phi\rangle=(\cosh t)^{-n-1}\langle\sigma_{\Lambda^\vee} (k_zb_t^{\pm1} k_z^{-1}k)\phi,\phi\rangle.
\]
Define
\begin{equation}
b_z= k_z b_t k_z^{-1}=\begin{pmatrix} (1_n-zz^*)^{-1/2} & 0 \\ 0 &
(1-z^*z)^{-1/2}\end{pmatrix}
\end{equation}
so that
\begin{equation}\label{wmc3}
\langle\omega(g)\phi,\phi\rangle=(\cosh t)^{-n-1}\langle \sigma_{\Lambda^\vee}(b_z^{\pm 1} k)\phi,\phi\rangle.
\end{equation}

\section{Zeta integrals} \label{s7}

We are ready to compute the archimedean zeta integrals on $U(n,1)$ that involves oscillator representations and holomorphic discrete series, combining the results in the previous sections.

In terms of the Harish-Chandra decomposition, by (\ref{dismc}) and (\ref{wmc3}) we have
\begin{align*}
\int_G \langle \omega(g)\phi,\phi\rangle \cdot \psi_\pi(g)dg&=\int_C \int_K
\langle \omega(h_zk)\phi,\phi\rangle\cdot\psi_\pi(h_zk)dk d^*z\\
&= \int_C\int_K  \langle \sigma_{\Lambda^\vee}(b_z^{\pm1}k)\phi, \phi\rangle\cdot
\textrm{tr }\sigma_\Lambda(\theta_z k) \det(1_n-zz^*)^{(n+1)/2}dk d^*z,
\end{align*}
noting that $\det(1_n-zz^*)=(\cosh t)^{-2}$.
We shall follow the strategy in \cite{G} to evaluate above integral, or more generally the integral
\begin{equation}\label{is}
I^\pm_s=\int_C\int_K  \langle \sigma_{\Lambda^\vee}(b_z^{\pm1}k)\phi, \phi\rangle\cdot
\textrm{tr }\sigma_\Lambda(\theta_z k) \det(1_n-zz^*)^s dk d^*z
\end{equation}
which converges absolutely for $\textrm{Re }s\gg0$. Here the $\pm$ sign is determined by $\Lambda^\vee$ as in Proposition \ref{6.1}, i.e. depends on whether we have Case (i) or (ii).

Our main result is the following

\begin{theorem} \label{main} Under the assumptions and notations of Sections \ref{s2} and \ref{s3}, for $\phi\in \mathcal{H}_{\Lambda^\vee, \Lambda'}$ and $\pi=\pi_\lambda$ one has the zeta integral

Case (i):
\[
\int_G \langle \omega(g)\phi,\phi\rangle \cdot \psi_\pi(g)dg=\frac{\pi^n}{\dim\sigma_\Lambda}\prod^n_{i=1}\frac{1}{\alpha_i-i+n}\|\phi\|^2;
\]

Case (ii):
\[
\int_G \langle \omega(g)\phi,\phi\rangle \cdot \psi_\pi(g)dg=\frac{\pi^n}{\dim \sigma_\Lambda}\prod^n_{i=1}\frac{1}{\gamma+i-p}\|\phi\|^2,\]
where $\dim\sigma_\Lambda$ is given by the well-known Weyl formula (\ref{dim}).
\end{theorem}

\begin{proof} Since $\sigma_{\Lambda^\vee}(b_z)^*=\sigma_{\Lambda^\vee}(b_z^*)=\sigma_{\Lambda^\vee}(b_z)$, $\sigma_\Lambda(\theta_z)^*=\sigma_\Lambda(\theta_z^*)=\sigma_\Lambda(\theta_z)$, we have
\begin{align*}
 \langle \sigma_{\Lambda^\vee}(b_z^{\pm1}k)\phi, \phi\rangle\cdot
\textrm{tr }\sigma_\Lambda(\theta_z k) &= \sum_i \langle \sigma_{\Lambda^\vee}(k)\phi,
\sigma_{\Lambda^\vee}(b_z^{\pm 1})\phi\rangle \cdot \langle \sigma_{\Lambda}(k)x_i,
\sigma_\Lambda(\theta_z)x_i\rangle
\end{align*}
where $\{x_i\}$ is an orthonormal basis of $\sigma_\Lambda$.
By Schur orthogonality relation, the integration over $K$ leaves us
\begin{align*}
I^\pm_s&=\frac{1}{\dim\sigma_\Lambda}\int_C\sum_i \langle\phi, x_i\rangle
\cdot \overline{\langle \sigma_{\Lambda^\vee}( b_z^{\pm 1})\phi, \sigma_\Lambda(\theta_z)x_i\rangle} \det(1_n-zz^*)^s d^*z\\
&=\frac{1}{\dim\sigma_\Lambda}\int_C\sum_i \langle\phi, x_i\rangle
\cdot \overline{\langle \sigma_{\Lambda^\vee}( \theta_z^{-1}b_z^{\pm 1})\phi, x_i\rangle} \det(1_n-zz^*)^s d^*z\\
&=\frac{1}{\dim\sigma_\Lambda}\int_C\langle\phi, \sigma_{\Lambda^\vee}(\theta_z^{-1} b_z^{\pm 1})\phi\rangle \det(1_n-zz^*)^s d^*z\\
&=\frac{1}{\dim\sigma_\Lambda}\langle\phi, (\int_C\sigma_{\Lambda^\vee}(\theta_z^{-1} b_z^{\pm 1}) \det(1_n-zz^*)^s d^*z)\phi\rangle.
\end{align*}
Hence we need to compute the endomorphism
\[
T^\pm_s=\int_C\sigma_{\Lambda^\vee}(\theta_z^{-1} b_z^{\pm 1}) \det(1_n-zz^*)^s d^*z\in\textrm{End}_{\bf C}(\sigma_{\Lambda^\vee}).
\]
We find that
\[
\theta_z^{-1} b_z=\begin{pmatrix} (1_n-zz^*)^{-1} & 0 \\  0 & 1\end{pmatrix},\quad \theta_z^{-1}b_z^{-1}=\begin{pmatrix} 1 & 0 \\ 0 & 1-z^*z\end{pmatrix}.
\]
If we decompose $\sigma_{\Lambda^\vee}\cong \sigma_1\otimes \sigma_2$ as the outer tensor product of irreducibles $\sigma_1$ of $\widetilde{U}(n)$ and $\sigma_2$ of $\widetilde{U}(1)$, then
\[
\left\{
\begin{aligned}
&T^+_s=\int_C \sigma^{-1}_1(1_n-zz^*)\det(1_n-zz^*)^s d^*z\in \textrm{End}_{\bf C}(\sigma_1),
\\
& T^-_s=\int_C \sigma_2(1-z^*z)\det(1_n-zz^*)^s d^*z\in \textrm{End}_{\bf C}(\sigma_2).
\end{aligned}
\right.
\]
By the parametrization of $D_{n,1}$ in Section \ref{s3}, a change of variables in the defining integral shows that $T_s^+$ commutes with $\sigma_1(k)$ for any $k\in \widetilde{U}(n)$, hence must be a scalar thanks to Schur's lemma. Similarly $T_s^-$ is a scalar as well. In other words, we have
\[
I_s^\pm=\frac{\overline{T_s^\pm}}{\dim\sigma_\Lambda}\|\phi\|^2=\frac{T_s^\pm}{\dim\sigma_\Lambda}\|\phi\|^2,
\]
noting that $T_s^\pm$ is real since the integrand is real. Recall the Weyl dimension formula
\begin{align}\label{dim}
\dim \sigma_\Lambda&=\dim \sigma_1 = \prod_{\alpha\in \Delta_c^+} \frac{\langle\Lambda+\rho_c, \alpha\rangle}{\langle\rho_c,\alpha\rangle}=\prod_{\alpha\in\Delta^+_c}\frac{\langle \lambda+\rho-\rho_c,\alpha\rangle}{\langle\rho_c,\alpha\rangle}\\
&=\prod_{\alpha\in\Delta^+_c}\frac{\langle\lambda,\alpha\rangle}{\langle\rho_c,\alpha\rangle}=\prod_{1\leq i<j\leq n} \frac{\lambda_i-\lambda_j}{j-i}.\nonumber
\end{align}

It remains to find the scalar $T_s^\pm$, which is essentially a special case of the computation in \cite{G}. However the second proposition in \cite[\S3]{G} was not stated correctly, which caused a mistake in the formulas of the main theorem therein. For reader's convenience a correct variant form is given in the Lemma \ref{gar} below.

Applying this lemma for the representations $\sigma_1\otimes 1$ and $1\otimes \sigma_2$  of $(U(n)\times U(1))^\sim$ respectively, we obtain

Case (i): $\Lambda^\vee=[(-\alpha_n,\ldots,-\alpha_1)+\textrm{det}^{(1-n)/2}]
\otimes [-\gamma+\textrm{det}^{-(1-n)/2}]$,
\[
T_s^+=S_{\sigma_1\otimes 1, s}=\pi^n\prod^n_{i=1}\frac{1}{\alpha_i-i+s-(1-n)/2};
\]

Case (ii): $\Lambda^\vee=[(\beta_1,\ldots,\beta_p, -\alpha_{q-1},\ldots,-\alpha_1)+\textrm{det}^{(p-q)/2}]
\otimes [\gamma+\textrm{det}^{-(p-q)/2}]$,
\[
T_s^-=S_{1\otimes \sigma_2, s}=\pi^n\prod^n_{i=1}\frac{1}{\gamma-i+s-(p-q)/2}.
\]
The theorem follows from specializing $s=(n+1)/2$.
\end{proof}

\begin{lemma} \label{gar} \cite{G}
Let $\sigma=\sigma_1\otimes\sigma_2$ be an irreducible representation of $(U(p)\times U(q))^\sim$, where $\sigma_1$ and $\sigma_2$ have highest weights $(\kappa_1,\ldots, \kappa_p)$ and $(\iota_{1},\ldots, \iota_{q})$ respectively. Define
\[
S_{\sigma,s}=\int_{D_{p,q}}\sigma_1^{-1}(1_p-zz^*)\otimes \sigma_2(1_q-z^*z)\det(1_p-zz^*)^s d^*z\in\textrm{End}_{\bf C}(\sigma)
\]
for $\textrm{Re }s\gg 0$, where
\[
D_{p,q}=\{z\in{\bf C}^{p\times q}: 1_p-zz^*\textrm{ is positive definite}\},\quad d^*z=\frac{dz}{\det(1_p-zz^*)^{p+q}}.
\]
Then $S_{\sigma,s}$ is a scalar and one has

(i) if $\sigma_2=\det^\iota$ is one-dimensional, then
\begin{align*}
S_{\sigma,s}&= \pi^{pq}\prod^p_{i=1}\frac{\Gamma(\iota-\kappa_i-(p+q-i)+s)}{\Gamma(\iota-\kappa_i-(p-i)+s)}\\
&=\pi^{pq}\prod^p_{i=1}\frac{1}{(\iota-\kappa_i-(p+q-i)+s)\cdots(\iota-\kappa_i-(p+1-i)+s)};
\end{align*}

(ii) if $\sigma_1=\det^\kappa$ is one-dimensional, then
\begin{align*}
S_{\sigma,s}&= \pi^{pq}\prod^q_{i=1}\frac{\Gamma(\iota_i-\kappa-(p+i-1)+s)}{\Gamma(\iota_i-\kappa-(i-1)+s)}\\
&=\pi^{pq}\prod^q_{i=1}\frac{1}{(\iota_i-\kappa-(p+i-1)+s)\cdots(\iota_i-\kappa-i+s)};
\end{align*}

(iii) if $\sigma$ is the lowest $\widetilde{K}$-type of an anti-holomorphic discrete series $\pi$ of $\widetilde{U}(p,q)$, then under above measure the formal degree of $\pi$ is given by
\[
\frac{1}{d_\pi} =\frac{S_{\sigma,0}}{\dim\sigma}.
\]
\end{lemma}

We remark that the formulation of \cite{G} is in terms of holomorphic discrete series, and our reformulation here about anti-holomorphic case is just for convenience.
Recall from \cite{HLS} that
\[
\int_G \langle \omega(g)\phi,\phi\rangle \cdot \psi_\pi(g)dg=\frac{c^2_{\psi,\pi}}{d_\pi}\|\phi\|^2,
\]
where $c_{\psi,\pi}$ is the positive number such that
\[
\| P_{\psi,\pi}(\phi)\|=c_{\psi,\pi}\|\phi\|
\]
for $\phi\in \mathcal{H}_{\Lambda^\vee,\Lambda'}$ and $P_{\psi,\pi}$ the orthogonal projection from $\omega_\psi$ onto the closed subspace $\sigma_{\Lambda^\vee}\otimes \pi'$. We are interested in the explicit value of $c_{\psi, \pi}$. The formal degree $d_\pi$ is given by (\ref{fd}), depending on the measure of $G$.  Instead of specifying the explicit dependence, we may compare our zeta integral with the formal degree given by Lemma \ref{gar} (iii). This will enable us to find out $c_{\psi, \pi}$.

\begin{corollary} \label{7.3}The explicit value of $c_{\psi,\pi}$ is given by

Case (i):
\[
c^2_{\psi,\pi}=\prod^n_{i=1}\frac{\alpha_i-i+n-1-\gamma}{\alpha_i-i+n};
\]

Case (ii):
\[
c^2_{\psi,\pi}=\prod^n_{i=1}\frac{\gamma+i-\delta_i-2p}{\gamma+i-p},
\]
where $(\delta_1,\ldots,\delta_n):=(\beta_1,\ldots,\beta_p,-\alpha_{q-1},\ldots,-\alpha_1)$.
\end{corollary}

\begin{proof} The proof of Theorem \ref{main} shows that
\[
\frac{c^2_{\psi,\pi}}{d_\pi}=\frac{T^\pm_{(n+1)/2}}{\dim \sigma_\Lambda}.
\]
On the other hand, by Lemma \ref{gar} we have
\[
\frac{1}{d_\pi}=\frac{S_{\sigma_{\Lambda^\vee,0}}}{\dim \sigma_\Lambda}.
\]
Comparison of the last two equations yields
\[
c^2_{\psi,\pi}=\frac{T^\pm_{(n+1)/2}}{S_{\sigma_{\Lambda^\vee,0}}}=\frac{S_{\sigma_1\otimes 1, (n+1)/2}}{S_{\sigma_1\otimes\sigma_2,0}}\quad \textrm{or} \quad
\frac{S_{1\otimes \sigma_2, (n+1)/2}}{S_{\sigma_1\otimes\sigma_2,0}}
\]
in Case (i) or (ii) respectively. Plugging in the parameter $\Lambda^\vee$ gives the corollary.
\end{proof}

\end{document}